\documentclass[a4paper, 10pt]{article}

\usepackage{amsmath, amscd, amsfonts, amssymb, amsthm, latexsym, url}
\usepackage{graphicx}
\usepackage{hyperref}
\input{xy}
\xyoption{all}

\usepackage{sectsty}
\sectionfont{\Large\bf}
\chapterfont{\large\bf\centering}
\chaptertitlefont{\LARGE\centering}
\partfont{\centering}
\paragraphfont{\bf}

\theoremstyle{plain}
\newtheorem{thm}{Theorem}[section]

\newtheorem{corollary}[thm]{Corollary}
\newtheorem{prop}[thm]{Proposition}

\theoremstyle{definition}
\newtheorem{df}[thm]{Definition}
\newtheorem{rmk}[thm]{Remark}

\newtheorem{exa}[thm]{Example}

\newcommand{\roundb}[1]{\! \left(\! \left( #1 \right)\!\right)}

\newcommand{\curlyb}[1]{\! \left\{\! \left\{ #1 \right\}\!\right\}}

\newcommand{\squarebs}[1]{\! \left[ #1 \right]}

\newcommand{\OO}{\mathcal{O}}
\newcommand{\pp}{\mathfrak{p}}

\newcommand{\spec}{\mathrm{Spec}\:}

\newcommand{\Ad}{\mathbb{A}}

\newcommand{\ZZ}{\mathbb{Z}}

\renewcommand{\hom}{\mathrm{Hom}}

\newcommand{\Top}{\mathbf{Top}}
\newcommand{\Seq}{\mathbf{Seq}}
\newcommand{\Sets}{\mathbf{Sets}}

\newcommand{\Alg}{\mathbf{Alg}}
\newcommand{\Sch}{\mathbf{Sch}}
\newcommand{\AffSch}{\mathbf{AffSch}}

\newcommand{\car}{\mathrm{char}\:}

\author{Alberto C\'amara\footnote{The author is supported by a Doctoral Training Grant at the University of Nottingham.}}
\title{Topology on rational points over higher local fields} 
\date{\today}

\begin{document}

\maketitle

\begin{abstract}
  We  develop a sequential-topological study of rational points of schemes of finite type over local rings typical in higher dimensional number theory and algebraic geometry. These rings are certain types of multidimensional complete fields and their rings of integers and include higher local fields. Our results extend the constructions of Weil over (one-dimensional) local fields. We establish the existence of an appropriate topology on the set of rational points of schemes of finite type over any of the rings considered, study the functoriality of this construction and deduce several properties. 
\end{abstract}

\section{Introduction}

The study of higher dimensional local fields was started in the 1970s by Parshin in positive characteristic and Kato in the general case. They are natural objects that arise in algebra and geometry by a process that iterates localization and completion. Let us recall the definition.

\begin{df}
  \label{df:hlfs}
  A zero dimensional local field is a finite field. For $n \geq 1$, an $n$-dimensional local field is a complete discrete valuation field, such that the residue field for the valuation is an $(n-1)$-dimensional local field.
\end{df}

In positive characteristic, the theory is also developed allowing zero dimensional local fields to be any perfect field of positive characteristic (cf. \textsection \ref{sec:otherhlfs}).

In the geometric context, these fields generalize the role played by local fields on curves over finite fields or, more generally, on global fields. The point of view is that instead of focusing on a point and a neighbourhood, we consider a maximal chain of subschemes ordered by inclusion.

The development of techniques to study the arithmetic of higher local fields, generalizing what is known for (one-dimensional) local fields, would have applications to the study of the arithmetic properties of higher dimensional schemes, including N\'eron models of elliptic curves. 

There are mainly four different approaches to the study of higher local fields: through topology; analysis and measure theory (see the works of Fesenko \cite{fesenko-aoas1}, Morrow \cite{morrow-integration-on-valuation-fields}, and Kim-Lee \cite{kim-lee-spherical-hecke-algebras}); category theory (see for example works by Beilinson \cite{beilinson-perverse-sheaves}, Drinfeld \cite{drinfeld-infinite-dimensional-vector-bundles}, Kapranov \cite{kapranov-double-affine-hecke-algebras} and Previdi \cite{previdi-locally-compact-objects-exact-cats}) and model theory(see for example works by Hrushovski-Kazhdan \cite{hrushovski-kazhdan-integration-in-valued-fields}). Finding relations among these points of view and unifying them is quite a challenging problem.

Regarding the topological approach, with which this work is concerned, the major achievement is local higher class field theory. This theory was established in the general case by Kato \cite{kato-lCFT-1}, \cite{kato-lCFT-2}, \cite{kato-lCFT-3}; Fesenko \cite{fesenko-cft-of-multidim-local-fields-char-zero}, \cite{fesenko-multidimensional-lt-of-cf}; Spiess \cite{spiess-class-formations-hlcft} and others. It uses Milnor $K$-groups endowed with a certain topology and the higher reciprocity map in order to describe abelian extensions of a higher local field. For a collection of surveys on higher local class field theory, see \cite{ihlf}.

We should remark that the valuation topology on a higher local field and the usual topology on its units do not suffice for this purpose, and the introduction of more sophisticated topologies is required. Working  simultaneously  with  several  topologies helps to study deep arithmetic properties in higher number theory. Of course, there are many other  similar examples, e.g. the use of six topologies in the work on the Bloch-Kato conjecture by Voevodsky. 

Such topologies are what we call {\it higher topologies}. When taken into account, multiplication on a higher local field of dimension greater than 1 ceases to be continuous, and hence we are no longer working in a category of topological fields. However, the fact that multiplication is sequentially continuous is enough to show many interesting properties of the arithmetic of these fields: Fesenko \cite{fesenko-abelianlocalpclass} seems to have been the first one to exploit the sequential properties of higher topologies in the study of higher class field theory.

One of the conclusions of this work is the utility of the sequential point of view on higher topology when it comes to endowing sets of rational points over higher local fields with a topology.

\vskip .5cm

This work is organised as follows. In \textsection \ref{sec:sectopgpsandrings}, we review basic facts about sequential topology and introduce a new concept: a sequential ring (Definition \ref{df:seqring}). \textsection \ref{sec:examples} provides the main examples of sequential rings. It consists of a small survey on higher topologies on higher local fields. Although these sections are mainly expository, several well-known facts about higher topology are stated and proved in \textsection \ref{sec:examples}; despite being known to experts in the area, it is very difficult to refer to a proof. In particular, (v) in Proposition \ref{prop:hdtop} and Theorem \ref{thm:norelationlinearseq} are new; we have not found these results in the existing literature.

In \textsection \ref{sec:ratpoints}, schemes of finite type over a sequential ring are considered. Under some conditions on the ring, we construct a topology on the set of rational points, and we study the properties of this construction. The key result in this direction is Theorem \ref{thm:topratptsgeneral}.

The main available examples of sequential rings are complete discrete valuation fields such that their residue fields are endowed with a topology that has certain properties. For such fields, we may apply a general construction that produces a topology which turns the field into a sequential ring.

Following the construction given by Theorem \ref{thm:topratptsgeneral}, we obtain a topology on the set of rational points on a scheme of finite type over a higher local field and also on schemes over the ring of integers of a higher local field. Several properties of this topology are deduced from this general setting. 

Finally, we discuss further work and connections to other topics in \textsection\ref{sec:future}. In particular, the possibility of applications of the results in this work to the study of higher adelic rational points of algebraic varieties and to the study of representation theory of algebraic groups over higher local fields ought to be emphasized.

\paragraph{Notation.} By the word \textit{ring} we will always mean a commutative, unital ring.

Throughout this text, if $F$ denotes a complete discrete valuation field, $\OO_F$ will denote the ring of integers for the unique normalized discrete valuation on $F$ and $\pp_F \subset \OO_F$ will denote the unique maximal ideal of $\OO_F$. By $\pi_F$ we will typically denote an element of $F$ such that $\pp_F = \pi_F \OO_F$.

We will denote $\overline{F} = \OO_F / \pp_F$ for the residue field of $F$ and $\rho_F: \OO_F \rightarrow \overline{F}$ will denote the residue homomorphism. If $F$ is a higher local field of dimension $n$ as in Definition \ref{df:hlfs}, the rank-$n$ ring of integers of $F$, defined in \textsection\ref{sec:examples}, will be denoted by $O_F$.

In sections \textsection\ref{sec:sectopgpsandrings} and \textsection \ref{sec:ratpoints}, $R$ and $S$ will denote sequential rings as in Definition \ref{df:seqring}.

In section \textsection \ref{sec:ratpoints}, $X$ and $Y$ will denote schemes, which we will always consider of finite type, over a sequential ring.

Throughout the text, $\Sets$ stands for the category of sets. When $R$ is a ring, $\Alg_R$ and $\Sch_R$ stand for the categories of finitely generated $R$-algebras and schemes of finite type over $R$, respectively. $\AffSch_R$ is the subcategory of $\Sch_R$ whose objects are affine schemes of finite type over $R$. Finally, $\Top$ is the category of topological spaces and $\Seq$ is the subcategory of sequential topological spaces.

When dealing with the elements of a sequence, the expression \textit{almost all} may be safely replaced by \textit{all but finitely many}.

\paragraph{Acknowledgement.} I am grateful to my advisor I. B. Fesenko for suggesting this direction of work, for many instructive conversations and for his useful guidance. I am also very grateful to M. T. Morrow and O. Br\"{a}unling, who during countless conversations have given invaluable advice and suggestions, have read early versions of this work and have pointed out many improvements.

\section{Sequential topology on groups and rings}
\label{sec:sectopgpsandrings}

We review a few aspects about the category of sequential topological spaces. Details and proofs for the statements in this section may be found in \cite{franklin-sequences1}, \cite{franklin-sequences2}.

Let $(X, \tau)$ be a topological space.

\begin{df}
  A sequence $(x_n)_n \subset X$ is convergent to $x \in X$ if for every open neighbourhood $U$ of $x$ in $X$, almost all of the elements in the sequence $(x_n)_n$ belong to $U$. That is: there exists $n_0 \in \mathbb{N}$ such that $x_n \in U$ for $n \geq n_0$.
\end{df}

We will use the notation $x_n \rightarrow x$ whenever $(x_n)_n \subset X$ is a sequence which converges to $x \in X$.

In general, specifying the set of convergent sequences of a set $X$ does not determine a unique topology on $X$, but rather a whole family of topologies. Among them, there is one which is maximal in the sense that it is the finest: the sequential saturation.

\begin{df}
  A subset $A$ of $X$ is sequentially open if for any sequence $(x_n)_n$ in $X$ convergent to $x \in A$, there is an index $n_0 \in \mathbb{N}$ such that $x_n \in A$ for $n \geq n_0$.
\end{df}

\begin{prop}
  Every open set is sequentially open. The collection of sequentially open sets of $X$ defines a topology $\tau_s$ on $X$, finer than $\tau$. \qed
\end{prop}

The natural map $(X, \tau_s) \to (X, \tau)$, given by the identity on the set $X$, is continuous. 

\begin{df}
  The space $(X,\tau_s)$ is called the sequential saturation of $(X,\tau)$. Whenever $\tau = \tau_s$, we shall simply say that $X$ is sequential.
\end{df}

\begin{rmk}
  A subset $C \subseteq X$ is sequentially closed if for every $x_n \to x$ in $X$, $x_n \in C$ for every $n$ implies $x \in C$. This condition is equivalent to $X \setminus C$ being sequentially open. We could have defined $\tau_{s}$ by specifying that its closed sets are all sequentially closed sets in $(X, \tau)$.
\end{rmk}

\begin{prop}
  Let $f: X \rightarrow Y$ be a map between topological spaces. The following are equivalent:
  \begin{enumerate}
    \item For any convergent sequence $x_n \rightarrow x$ in $X$, the sequence $(f(x_n))_n$ converges to $f(x)$ in $Y$.
    \item The preimage under $f$ of any sequentially open set in $Y$ is sequentially open.
    \item The map $f: (X,\tau_s) \rightarrow Y$ is continuous.
  \end{enumerate}
  When this situation holds, we say that $f$ is a sequentially continuous map. \qed
\end{prop}

Any continuous map is sequentially continuous. Hence, sequential continuity is a weaker condition than ordinary continuity. In a sequential space, the topology is essentially controlled by sequences.

\begin{exa}
  Any metric space is a sequential space. More generally, any first countable space is sequential \cite[\textsection 1]{franklin-sequences1}.
\end{exa}

\begin{exa}
  The Stone-\v{C}ech compactification of a topological space $X$ is the unique Hausdorff compact space $\beta X$ provided with a continuous map $X \to \beta X$ such that for any Hausdorff compact space $Y$ and any continuous map $f: X \to Y$ there is a unique map $\beta f: \beta X \to Y$ such that the diagram
  \begin{equation*}
    \xymatrix{
    X \ar[r]^f \ar[d] & Y \\
    \beta X \ar[ru]_{\beta f}
    }
  \end{equation*}
  commutes. The Stone-\v{C}ech compactification of the natural numbers, $\beta \mathbb{N}$, is an example of a space which is not sequential \cite[Example 1.1]{goreham-sequential-convergence-in-top-spaces}. Despite being compact, $\beta \mathbb{N}$ is not sequentially compact. In this space, the closure of a set does not consist only of the limits of all sequences in that set, and a function from $\beta \mathbb{N}$ to another topological space may be sequentially continuous but not continuous.
\end{exa}

Let $\Seq$ denote the subcategory of sequential topological spaces. Taking the sequential saturation of a topological space defines a functor $\Top \rightarrow \Seq$.

Some of the usual constructions in $\Top$ are not inherited by $\Seq$. The product of topological spaces is a remarkable example of such failure. Other examples of operations which do not behave well with respect to sequential saturation are function spaces and subspaces \cite[Example 1.8]{franklin-sequences1}. However, open and closed subspaces and open and closed images are closed constructions in $\Seq$ \cite[Prop. 1.9]{franklin-sequences1}.

These sort of problems may be addressed by performing the usual construction in $\Top$, and then taking the sequential saturation of the resulting space. In this fashion, the sequential saturation of the product topology provides a product object in $\Seq$. In the words of Steenrod \cite{steenrod-convenient-category-topological-spaces}, $\Seq$ is a {\it convenient} category of topological spaces.

\vskip .5cm

We are interested in the compatibility between sequential topology and algebraic structures.

\begin{df}
  A sequential group is a group $G$ provided with a topology, such that multiplication $G \times G \rightarrow G$ and inversion $G \rightarrow G$ are sequentially continuous ($G \times G$ is provided with the product topology). A homomorphism of sequential groups is a sequentially continuous group homomorphism.
  In other words: if $(G,\tau)$ is a sequential group, then $(G,\tau_s)$ is a group object in $\Seq$.
\end{df}

\begin{rmk}
  $(G, \tau_s)$ is also a group object in $\Top$. If $(G,\tau)$ is a topological group, then so is $(G,\tau_s)$.
\end{rmk}

When considering rings, we could be interested in topologies for which subtraction and multiplication are sequentially continuous. However, we will deal with rings and topologies on them such that their additive groups are topological groups. 

\begin{df}
  \label{df:seqring}
  A sequential ring is a commutative ring $R$ provided with a topology and such that:
  \begin{enumerate}
    \item $(R,+)$ is a topological group.
    \item Multiplication $R\times R \rightarrow R$ is sequentially continuous.
  \end{enumerate}
  A homomorphism of sequential rings is a continuous ring homomorphism.
\end{df}

\begin{rmk}
  We could have required homomorphisms of sequential rings to be sequentially continuous maps. However, there is a good reason to prefer this stronger condition: in this way we preserve the underlying topological structure of the additive groups. Still, a continuous homomorphism of sequential rings $R \to S$ does not furnish $S$ with the structure of a topological $R$-module. This is why we consider the notion of sequential module.
\end{rmk}

\begin{df}
  Let $M$ be an $R$-module. We say that $M$ is a sequential module if $M$ is provided with a topology such that $(M,+)$ is a topological group and multiplication 
  \begin{equation*}
    R \times M \rightarrow M
  \end{equation*}
  is sequentially continuous.
  A homomorphism of sequential $R$-modules is a continuous $R$-module homomorphism.
\end{df}
Note that with this definition $R$ is a sequential $R$-module. If $R \rightarrow S$ is a homomorphism of sequential rings, $S$ is automatically endowed with the structure of a sequential $R$-module.

\vskip .5cm

When $R$ is a topological ring, it is always possible to provide the units with a group topology. Because we do not demand inversion to be continuous, the correct way to topologize $R^\times$ is by considering the initial topology for the map
\begin{equation}
  \label{eqn:topringembedintotwocopies}
  R^\times \rightarrow R \times R, \quad x \mapsto \left( x, x^{-1} \right).
\end{equation}

The situation when $R$ is a sequential ring is not very different, since we do not demand inversion on $R^\times$ to be sequentially continuous.  The topology on $R^\times$ given by the sequential saturation of the initial topology for the map (\ref{eqn:topringembedintotwocopies}) turns $R^\times$ into a topological group.

\begin{rmk}
  A priori, we do not require sequential groups, rings and modules to be sequential topological spaces.
\end{rmk}

\section{Examples of sequential groups and rings}
\label{sec:examples}

Of course, any topological group (resp. ring) is an example of a sequential group (resp. ring). We will consider other objects provided with a sequential structure which is not a topological one in the usual sense.

\subsection{Higher local fields}
\label{subsec:HLFS}

We have introduced higher local fields in Definition \ref{df:hlfs}. An alternative to this definition: a higher local field is a sequence of fields $F = F_n, F_{n-1},\dots, F_0$ where $F_0$ is finite and for $1 \leq i \leq n$, $F_i$ is a complete discrete valuation field with residue field $F_{i-1}$. See \cite[Definition 2.1.1]{yekutieli-explicit-construction-grothendieck-residue-complex} for yet another equivalent definition.

\begin{exa}
  \label{exa:equi2dlf}
  The fields $\mathbb{F}_q \roundb{t}\roundb{u}$ and $\mathbb{Q}_p \roundb{t}$ are both examples of higher local fields of dimension two.
\end{exa}
 
The example below is slightly more difficult, but very important nonetheless.

\begin{exa}
  Consider the field
  \begin{equation*}
    \mathbb{Q}_p \curlyb{t} = \left\{  \sum_{i=-\infty}^{\infty} a_i t^i, a_i \in \mathbb{Q}_p, \inf_{i}v_p(a_i) > -\infty, \lim_{i\rightarrow -\infty} a_i = 0 \right\},
  \end{equation*}
  with operations given by the usual addition and multiplication of power series. Note that we need to use convergence of series in $\mathbb{Q}_p$ in order to define the product. With the discrete valuation given by
  \begin{equation*}
    v\left( \sum_{i=-\infty}^{\infty} a_i t^i \right) := \inf_i v_p(a_i),
  \end{equation*}
  the field $\mathbb{Q}_p \curlyb{t}$ turns into a 2-dimensional local field, its first residue field being $\mathbb{F}_p \roundb{t}$. This field may also be obtained as the completion of $\mathbb{Q}_p \squarebs{t}$ with respect to the gaussian valuation.
\end{exa}

For an introduction to the topic of higher local fields and a collection of surveys on their structure and the structure of their extensions, see \cite{ihlf}. For the details of the topological issues we will discuss in this section, see \cite{madunts-zhukov-topology-hlfs}.

\vskip .5cm

Besides the ring of integers $\OO_F$ for the unique normalized discrete valuation of $F$, we might consider other valuation rings in $F$ which are of arithmetic interest. Fix a system of parameters $t_1, \ldots, t_n$ of $F$. That is, $t_n$ is a uniformizer of $F$, $t_{n-1}$ is a unit of $\OO_F$ such that its image in $\overline{F}$ is a uniformizer, and so on.

We may use our chosen system of parameters to define a valuation of rank $n$:
\begin{equation*}
  \mathbf{v} = (v_1, \ldots, v_n): F^\times \rightarrow \ZZ^n,
\end{equation*}
where $\ZZ^n$ is ordered with the inverse of the lexicographical order. The procedure we follow is: $v_n = v_F$, $v_{n-1}(\alpha) = v_{\overline{F}}(\alpha t_n^{-v_n(\alpha)})$, and so on.

Although $\mathbf{v}$ does depend on the choice of the system of parameters, the valuation ring
\begin{equation*}
  O_F = \left\{ \alpha \in F;\; \mathbf{v}(\alpha) \geq 0 \right\}
\end{equation*}
does not. The ring $O_F$, which is a local ring with maximal ideal $\left\{ \alpha;\;\mathbf{v}(\alpha) > 0 \right\}$, is called the rank-$n$ ring of integers of $F$. It is a subring of $\OO_F$.

Once a system of parameters is chosen, we may consider a valuation of rank $r$ for every $1 \leq r \leq n$ by mimicking the same procedure and stopping after $r$ steps. The valuation rings we obtain are independent of any choice of parameters and are ordered by inclusion:
\begin{equation}
  \label{eqn:higherranksubrings}
  O_F = O_1 \subset O_2 \subset \cdots \subset O_n = \OO_F.
\end{equation}

\vskip .5cm

A higher local field has a topology determined by its unique normalized discrete valuation. This topology turns $F$ into a Hausdorff and complete topological group which is not locally compact if $n >1$, due to the fact that the residue field is not finite. There are several reasons to consider a different topology, such as the fact that the formal series in fields such as the ones in Example \ref{exa:equi2dlf} do not converge in the valuation topology. 

A higher topology takes into account the topology of the residue field and has some good properties, but in general it is not a ring topology. While the additive group is always a topological group with respect to this topology, multiplication is only sequentially continuous. It is for this reason that sequential topology plays an important role in the theory of higher local fields.

Let $E$ be a field provided with a topology such that addition and multiplication by a fixed element are continuous. Consider the field of Laurent power series $E \roundb{t}$. We define a topology as follows.

Let $\left\{ U_i \right\}_{i \in \mathbb{Z}}$ be a sequence of open neighbourhoods of zero of $E$, with the property that there is an $i_0 \in \mathbb{Z}$ such that $U_i = E$ for $i \geq i_0$. Define
\begin{equation}
  \label{eqn:definebasicopenhighertop}
  \mathcal{U} = \left\{ \sum_{i \gg -\infty} f_i t^i\in E\roundb{t};\; f_i \in U_i \right\}.
\end{equation}

\begin{prop}
  \label{prop:liftoftop}
  Let $E$ be a field endowed with a topology such that addition and multiplication by a fixed element are continuous. The collection of sets defined by (\ref{eqn:definebasicopenhighertop}) is a basis of open neighbourhoods of zero for a group topology on $E\roundb{t}$. 
\end{prop}

A similar procedure can be applied to higher local fields. If $\car F = \car \overline{F}$, then an isomorphism $F \simeq \overline{F}\roundb{t}$ is available. The mixed characteristic case, in which $\car F \neq \car \overline{F}$, is more complicated. 

Assume first that $F$ is a standard mixed-characteristic higher local field, i.e; that there is a local field $K$ and an isomorphism $F = K \curlyb{t_1} \cdots \curlyb{t_{n-1}}$ (note that for this field we may choose $\pi_F = \pi_K$), and assume that the higher topology has already been constructed for $F' = K \curlyb{t_1}\cdots\curlyb{t_{n-2}}$, which is an $(n-1)$-dimensional local field.

The choice of local parameters $t_1 = \pi_F, t_2, \ldots, t_n$ determines a {\it canonical} lifting \footnote{a set-theoretic section of the residue homomorphism with some special properties, see \cite[Lemma 1.2]{madunts-zhukov-topology-hlfs} for the definition. Indeed, obtaining this {\it canonical} lifting is the most difficult step in the construction of a higher topology}
\begin{equation*}
  h: \overline{F} \rightarrow \OO_F.
\end{equation*}

Assume that the topology on $\overline{F}$ has already been constructed. Let $\left\{ U_i \right\}_{i \in \mathbb{Z}}$ be a sequence of open neighbourhoods of zero in $\overline{F}$, and assume that there is an index $i_0 \in \mathbb{Z}$ such that $U_i = \overline{F}$ for all $i \geq i_0$. Define
\begin{equation}
  \label{eqn:definebasicopenhighertopmixed}
  \mathcal{U} = \left\{ \sum_{i \gg -\infty} h(f_i) \pi_F^i \in F;\; f_i \in U_i, \text{ for all } i \gg -\infty \right\}.
\end{equation}
The sets of the form (\ref{eqn:definebasicopenhighertopmixed}) define the basis of open neighbourhoods of zero for a group topology on $F$.

In the general case in which $F$ is a nonstandard mixed-characteristic local field, we may find a finite subextension $M \subset F$ which is a standard mixed-characteristic local field, apply the previous construction to $M$ and provide $F$ with the product topology through the linear isomorphism $F \simeq M^{\left[ F:M \right]}$.

\vskip .5cm

By iterating the above constructions and using classification results for higher local fields and choices of parameters we may construct a higher topology on a general $n$-dimensional local field. Details (and proofs) may be found in \cite[\textsection 1.1 -- \textsection 1.4]{madunts-zhukov-topology-hlfs}.

\begin{df}
  Any topology constructed on an $n$-dimensional local field using the above construction is called a higher topology.
\end{df}

\begin{rmk}
  In general, a higher topology on $F$ depends on the choice of a system of parameters. However, we should remark that whenever $\car \overline{F} = p$, the topology depends neither on the choice of a system of parameters or a standard subfield $M \subset F$ in the mixed-characteristic case \cite[Theorem 1.3]{madunts-zhukov-topology-hlfs}. In the case $\car F = \car \overline{F} = 0$, the topology does not depend on the choice of a parameter, but it depends on the choice of a coefficient field \cite[\textsection 1.4]{madunts-zhukov-topology-hlfs}.

  When an isomorphism such as $F \simeq \overline{F}\roundb{t}$ or $F \simeq F'\curlyb{t}$ has been fixed, we will drop the indefinite article and refer to {\it the} higher topology.
\end{rmk}

We summarize the properties of higher topologies below.

\begin{prop}
  \label{prop:hdtop}
  Let $F$ be an $n$-dimensional local field. Then any higher topology on $F$ satisfies the following properties:
  \begin{enumerate}
    \item $(F,+)$ is a topological group which is complete and Hausdorff.
    \item If $n > 1$, every base of neighbourhoods of zero is uncountable.
    \item If $n > 1$, multiplication $F \times F \rightarrow F$ is not continuous. However, both multiplication $F \times F \to F$ and inversion $F^\times \to F^\times$ are sequentially continuous.
    \item Multiplication by a fixed nonzero element $F \rightarrow F$ is a homeomorphism.
    \item The residue homomorphism $\rho: \OO_F \rightarrow \overline{F}$ is open.
  \end{enumerate}
\end{prop}

\begin{proof}
  See \cite[\textsection 1.3.2]{ihlf} for parts {\it (i)} to {\it (iv)}. Regarding {\it (v)}, the statement is a well known fact but unfortunately seems to be unavailable in the literature; we shall provide a proof for completeness.
 
  First assume that $F \simeq \overline{F}\roundb{t}$ is of equal-characteristic. In such case, let $U_0$ be an open neighbourhood of zero in $\overline{F}$. Then 
  \begin{equation*}
    \rho^{-1}\left( U_0 \right) = U_0 + \sum_{i \geq 1} \overline{F} t^i = U_0 + t \OO_F,
  \end{equation*}
  which is open in $F$. Moreover, an open neighbourhood of zero $\mathcal{U} \subseteq \OO_F$ is of the form $\mathcal{U} = \sum_{i\geq 0} U_i t^i$ where $U_i \subseteq \overline{F}$ are open neighbourhoods of zero, and we have $\rho\left( \mathcal{U} \right) = U_0$.

  Second, assume that $F$ is of mixed-characteristic. Without loss of generality, we may assume that $F \simeq F'\curlyb{t}$ is standard, as openness of maps is preserved on finite cartesian products. Let $U_0 \subseteq \overline{F}$ be an open neighbourhood of zero. Then $\rho^{-1}\left( U_0 \right) = h(U_0) + \pi_F \OO_F$, which is open. Finally, if $\mathcal{U}$ is as in (\ref{eqn:definebasicopenhighertopmixed}), then $\rho(\mathcal{U}) = U_0$, which is open.
\end{proof}

\begin{rmk}
  Failure of a higher topology on $F$ to be a ring topology is the original reason that led to the consideration of sequential rings in this work.
\end{rmk}

A group topology on an abelian group is said to be linear if the filter of neighbourhoods of the identity element admits a collection of subgroups as a basis. A commutative ring $R$ provided with a topology $\tau$ which is linear for the additive group and for which multiplication maps
\begin{equation*}
  R \to R,\quad y \mapsto xy
\end{equation*}
are continuous for all $x \in R$ is said to be a semitopological ring.

In view of (i) and (iv) in Proposition \ref{prop:hdtop}, it is possible to show that a higher local field endowed with a higher topology is a semitopological ring. A theory of semitopological rings has been developed and applied to the study of higher fields by Yekutieli \cite{yekutieli-explicit-construction-grothendieck-residue-complex} and others.

There seems to be a disagreement between the sequential and linear approaches to topologies on higher local fields.

\begin{thm}
  \label{thm:norelationlinearseq}
  Let $F$ be a higher local field. Denote $\tau$ for the topology on $F$ defined by Proposition \ref{prop:hdtop}, and let $\tau_s$ be its sequential saturation. The collections of open subgroups for $\tau$ and $\tau_s$ agree.
\end{thm}

\begin{proof}
We split the proof in several steps. Let $F_i = t^i \OO_F$.

\textit{Step 1. A subset $A \subseteq F$ is sequentially open if and only if $A \cap F_i$ is sequentially open in $F_i$ for every $i$}.

Suppose that $A \cap F_i$ is sequentially open in $F_i$ for every $i$. Let $x_n \to x \in A$. Then, there is an index $j$ such that $x_n \in F_j$ for all $n$, and since $F_j$ is sequentially closed, $x \in F_j$. As $A \cap F_j$ is sequentially open and $x \in A \cap F_j$, almost all of the $x_n$ belong to $A \cap F_j \subset A$, showing that $A$ is sequentially open.

\textit{Step 2. Let $Y \subset F_i$ be a subset such that $0 \in Y$. $Y$ is sequentially open if and only if it contains a subgroup $\mathcal{U} = \sum_{j \geq i} U_j t^j$, where the residues of elements in $U_j$ are open subgroups of $\overline{F}$ and $U_j = \overline{F}$ if $j$ is large enough}.

Let $u$ be a lift of a uniformizer of $\overline{F}$. If no such subgroup $\mathcal{U}$ is contained in $Y$, then for every $n$ there is an element $y_n \in t^i\left( u^n O_F + t^n \OO_F \right)$. However, $y_n \to 0 \in Y$ and therefore almost all of the $y_n$ must belong to $Y$, a contradiction.

\textit{Step 3. A subgroup of $H$ of $F$ is open for $\tau$ if and only if it is open for $\tau_s$}.

Suppose that $H$ is open for $\tau_s$. By step 1, $H \cap F_i$ is sequentially open and by step 2 it contains a subset $U_i t^i$, such that the image of $U_i$ in $\overline{F}$ is open and it contains $F_n$ for some $n$.

Put $\mathcal{U} = \sum U_i t^i \cap F$.

If $\car F = \car \overline{F}$, $U_i$ may be chosen so that it is an open subgroup of $\overline{F}$ viewed inside $F$. Otherwise, $U_i$ may be modified to ensure that $\mathcal{U}$ is a subgroup of $H$ such that the image of $U_i$ in $\overline{F}$ is an open subgroup and $\mathcal{U}$ contains some $F_n$. In both cases, $\mathcal{U}$ is an open subgroup for $\tau$.

Since $H$ is the union of $\mathcal{U}$-cosets, $H$ is also an open subgroup for $\tau$.
\end{proof}

\begin{corollary}
  For a general higher local field $F$ and a higher topology $\tau$, we have that $\tau_s$ is not a linear topology.
\end{corollary}

\begin{proof}
  A linear topology is completely determined by its collection of open subgroups. After the previous theorem, it suffices to show that $\tau \neq \tau_s$. So we recover a counterexample from \cite{fesenko-sequential-topologies}.
  
  Let $F = \mathbb{F}_p \roundb{u} \roundb{t}$. Let $C = \left\{ t^a u^{-c} + t^{-a} u^c,\; a, c \geq 1 \right\}$. Then $W = F \setminus C$ is open for $\tau_s$.

  Suppose that $U_i \subset \mathbb{F}_p \roundb{u}$ are open subgroups such that $U_i = \mathbb{F}_p \roundb{u}$ if $i$ is large enough and such that $\mathcal{U} = \sum U_i t^i \cap F$ is contained in $W$. Then, for any positive $c$ such that $u^c \in U_{-a}$ we would have $t^a u^{-c} + t^a u^{c} \in W$, a contradiction. Hence, $W$ is not open for $\tau$.
\end{proof}

Although the approach to higher topologies by using linear topologies and semitopological rings is useful for the study of the construction of higher adeles by means of ind-pro functors, the sequential approach is very important from the point of view of higher class field theory. When dealing with rational points over higher local fields, the sequential approach will allow us to say something about the continuity of polynomial maps $R^n \to R$, whereas this is not possible with a semitopological ring.

\subsection{The unit group of a higher local field}

There are several approaches to constructing topologies on $F^\times$. Once we know that $F$ is a sequential ring, and that inversion is sequentially continuous, a natural definition is the following.

\begin{df}
  \label{df:toponunitsHLF}
  The topology we consider on $F^\times \subset F$ is the sequential saturation of the subspace topology. 
\end{df}

\begin{rmk}
  Compare with (\ref{eqn:topringembedintotwocopies}); in the case of a higher local field, taking the embedding into two copies of $F$ is unnecessary, as we know that inversion is sequentially continuous.
\end{rmk}

The proof of the result below is obvious.

\begin{prop}
  The topology on $F^\times$ given by Definition \ref{df:toponunitsHLF} is a group topology; $F^\times$ is a Hausdorff complete group.
\end{prop}

\vskip .5cm

Assume, until the end of this paragraph, that $\car F > 0$. In this case, Parshin suggested a different approach, which was later refined by Yekutieli \cite[\textsection 3]{yekutieli-explicit-construction-grothendieck-residue-complex}. After choosing parameters, we obtain a decomposition
\begin{equation}
  \label{eqn:decomposeunitsofhlf}
  F^\times \simeq \mathbb{Z} t_n \oplus \cdots \oplus \mathbb{Z} t_1 \oplus F_0^\times \oplus V_F,
\end{equation}
where $F_0$ is the last residue field and $V_F$ is the group of principal units. $F^\times$ is provided with a topology by (\ref{eqn:decomposeunitsofhlf}) once $\mathbb{Z} t_n \oplus \cdots \oplus \mathbb{Z} t_1 \oplus F_0^\times$ is given the discrete topology and $V_F \subset \OO_F$ the subspace topology for the higher topology; this is called the Parshin topology.

We summarize the properties of this topology on units below.

\begin{prop}
  \label{prop:propsoftoponunits}
  Suppose $\car F > 0$. The topology we have defined on $F^\times$ satisfies:
  \begin{enumerate}
    \item The group $F^\times$ is a topological group with respect to the Parshin topology only when $F$ is of dimension at most 2.
    \item The Parshin topology is weaker than the valuation topology on $F^\times$.
    \item The Parshin topology does not depend on the choice of a system of parameters.
    \item $F^\times$ is complete.
    \item Multiplication on $F^\times$ is sequentially continuous.
  \end{enumerate}
\end{prop}

\begin{proof}
  See \cite[\textsection 3]{madunts-zhukov-topology-hlfs} for {\it (i)} to {\it (iii)}, \cite[\textsection 1.4]{ihlf} for the rest.
\end{proof}

\begin{prop}
  Suppose $\car F > 0$. The sequential saturation of the Parshin topology on $F^\times$ agrees with the sequential saturation of the topology on $F^\times$ defined by (\ref{eqn:topringembedintotwocopies}).
\end{prop}

\begin{proof}
  Denote the Parshin topology by $\tau$, and the initial topology defined by (\ref{eqn:topringembedintotwocopies}) by $\lambda$. Although this is implicitly included in \cite{fesenko-sequential-topologies}, we describe the explicit argument here.

  A sequence $a_m=t_n^{i_{n,m}} \cdots t_1^{i_{1,m}} u_m$ of elements of $F^\times$, with $u_m$ in the unit group of the ring of integers of $F$ with respect to any of its discrete valuations of rank $n$ tends to $1$ in $\lambda$ if and only if the sequence $u_m -1$ tends to $0$ with respect to the higher dimensional topology on $F$ (described by Proposition \ref{prop:hdtop}) and for every $j$ such that $1 \leq j \leq m$, the sequence $i_{j,m}$ is constant for sufficiently large $m$. But this last condition is equivalent to the sequence $a_m$ converging to $1$ with respect to $\tau$.
\end{proof}

\begin{exa}
  The $n$-th Milnor $K$-group of any field $F$ is presented as the term in the right in the following exact sequence of abelian groups:
  \begin{equation*}
    0 \rightarrow I_n \rightarrow F^\times \otimes_\ZZ \stackrel{(n)}{\cdots} \otimes_\ZZ F^\times \rightarrow K_n(F) \rightarrow 0,
  \end{equation*}
  where $I_n = \langle a_1 \otimes \cdots\otimes a_n;\; a_i + a_j = 1 \text{ for some } i \neq j \rangle_\ZZ$. By definition, $K_1(F) = F^\times$ and by convention $K_0(F) = \ZZ$.

When $F$ is an $n$-dimensional local field, $K_n$ generalises the role of $K_1 = \mathbb{G}_m$ in describing abelian extensions of the field \cite{ihlf}. From this point of view, it provides a correct higher dimensional generalization of the group of units.
The functor $K_n$ is not representable for $n \geq 2$, meaning that in general $K_n(F)$ is not the set of $F$-rational points on any scheme.

Let $F$ be an $n$-dimensional local field, and consider the topology on $F$ (resp. $F^\times$) given by Proposition \ref{prop:hdtop} (resp. \ref{prop:propsoftoponunits}).

Consider the finest topology on $K_m(F)$ for which:
\begin{enumerate}
  \item The symbol map $F^\times \times \cdots \times F^\times \rightarrow K_m(F)$ is sequentially continuous.
  \item Subtraction $K_m(F) \times K_m(F) \rightarrow K_m(F)$ is sequentially continuous.
\end{enumerate}
This topology is sequentially saturated \cite[\textsection 4, Remark 1]{fesenko-sequential-topologies}.

The topological $m$-th Milnor $K$-group is
\begin{equation}
  \label{eqn:deftopmilnorgps}
  K_m^{\mathrm{t}}(F) = K_m(F) / \Lambda_m(F),
\end{equation}
where $\Lambda_m(F)$ is the intersection of all open neighbourhoods of zero.

A sequentially continuous Steinberg map $F^\times \otimes_\mathbb{Z} \cdots \otimes_\mathbb{Z} F^\times \rightarrow G$ where $G$ is a Hausdorff topological group induces a continuous homomorphism $K_m^{\mathrm{t}}(F) \rightarrow G$. The Artin-Schreier-Parshin pairing, the Vostokov pairing and the tame symbol are examples of such continuous homomorphisms defined on $K_m^{\mathrm{t}}(F)$ \cite[I.6.4]{ihlf}.
\end{exa}

\subsection{Other higher complete fields}
\label{sec:otherhlfs}

There are other fields which can be defined in a similar way as a higher local field, and for which the construction of a higher topology is still valid. We give here two important examples of such. Let $k$ be any perfect field.

\begin{df}
  A zero dimensional complete field is a perfect field. An $n$-dimensional complete field is a complete discrete valuation field such that its residue field is an $(n-1)$-dimensional complete field. If $F$ is an $n$-dimensional complete field with last residue field $k$, we say that $F$ is an $n$-dimensional complete field over $k$.
\end{df}

\begin{exa}
  Higher local fields coincide with higher dimensional complete fields over a finite field.
\end{exa}

Higher dimensional complete fields over arbitrary perfect fields may be classified using the same techniques we have discussed for higher local fields \cite[Theorem in \textsection 0]{madunts-zhukov-topology-hlfs}. For example, if $\car F > 0$ and the last residue field is $k$, we may choose a system of parameters $t_1, \ldots, t_n$ such that
\begin{equation*}
  F \simeq k\roundb{t_1} \cdots \roundb{t_n}.
\end{equation*}

The reason why we can consider topological issues in this more general setting is because the construction of a higher topology does not use at any point the fact that the last residue field is finite.

In \cite{fesenko-abelianlocalpclass}, Fesenko considered the class field theory of $n$-dimensional complete fields over a perfect field of positive characteristic. In particular, \cite[\textsection 2]{fesenko-abelianlocalpclass} contains a brief description on how to topologize such fields. The main result of interest, which completely resembles Proposition \ref{prop:hdtop}, is the following.

\begin{prop}
  An $n$-dimensional complete field $F$ over a perfect field of positive characteristic may be topologized in such a way that:
  \begin{enumerate}
    \item $(F,+)$ is a topological group which is complete and Hausdorff.
    \item If $n > 1$, every base of open neighbourhoods of zero is uncountable.
    \item Multiplication $F \times F \to F$ and inversion $F^\times \to F^\times$ are sequentially continuous.
    \item Multiplication by a fixed non-zero element $F \to F$ is a homeomorphism.
    \item The residue homomorphism $\rho: \OO_F \to \overline{F}$ is open.
  \end{enumerate}
  Moreover, the topology does not depend on the choice of parameters.
\end{prop}

\begin{proof}
  Consider the last residue field of $F$ as a topological field with respect to the discrete topology, and apply the inductive process described in \textsection \ref{subsec:HLFS}.
\end{proof}

Hence, these are more general examples of sequential rings.

\vskip .5cm

A notion of higher archimedean local field exists, and these may be topologized applying the same tools.

\begin{df}
  The only one-dimensional archimedean local fields are $\mathbb{R}$ and $\mathbb{C}$.
  Let $n > 1$. An $n$-dimensional archimedean local field is a complete discrete valuation field such that its residue field is a $(n-1)$-dimensional archimedean local field.
\end{df}

Let $\mathbb{K}$ be either $\mathbb{R}$ or $\mathbb{C}$.

\begin{rmk}
  The dimension of an archimedean higher local field does not agree with its dimension as a higher complete field over $\mathbb{K}$: an $n$-dimensional archimedean local field is an $(n-1)$-dimensional complete field over $\mathbb{K}$. The reason for this shift in the dimension is because, since $\mathbb{K}$ is itself already a local field, we want to put fields such as $\mathbb{R}\roundb{t}$ and $\mathbb{Q}_p \roundb{t}$ in the same box; both fields are two-dimensional with the definitions we have taken.
\end{rmk}

As the field $\mathbb{K}$ is of characteristic zero, higher archimedean local fields are easily classified: they are isomorphic to Laurent power series fields over $\mathbb{K}$.

\begin{prop}
  Let $F$ be an $n$-dimensional archimedean local field. Then, there are parameters $t_2,\ldots,t_n \in F$ such that
  \begin{equation*}
    F \simeq \mathbb{K}\roundb{t_2}\cdots\roundb{t_n}.
  \end{equation*}
\end{prop}

\begin{proof}
  Since $\car \mathbb{K} = 0$, all the residue fields of $F$ have characteristic zero; the result follows by induction in the dimension and by Cohen structure theory for equidimensional complete fields.
\end{proof}

Regarding the way to topologize such fields; we will always consider the euclidean topology on $\mathbb{K}$. This topology satisfies the conditions of Proposition \ref{prop:liftoftop}, and it is Hausdorff and complete. We apply the construction specified in the aforementioned Proposition inductively, in order to obtain the following result.

\begin{prop}
  An $n$-dimensional archimedean local field $F$ may be topologized in such a way that:
  \begin{enumerate}
    \item $(F,+)$ is a topological group which is complete and Hausdorff.
    \item If $n > 1$, every base of open neighbourhoods of zero is uncountable.
    \item Multiplication $F \times F \to F$ and inversion $F^\times \to F^\times$ are sequentially continuous.
    \item Multiplication by a fixed non-zero element $F \to F$ is a homeomorphism.
    \item The residue homomorphism $\rho: \OO_F \to \overline{F}$ is open.
  \end{enumerate}
\end{prop}

\begin{rmk}
  In the case of archimedean higher local fields, a higher topology does depend on the choice of a system of parameters.
\end{rmk}

\begin{rmk}
  Higher rank valuation rings (\ref{eqn:higherranksubrings}) may also be defined in the case of archimedean higher local fields and higher complete fields. In the first case, the system of parameters has $n-1$ elements and we can only construct a chain of $n-1$ subrings by this procedure.
\end{rmk}

\section{Rational points over sequential rings}
\label{sec:ratpoints}

\subsection{Affine case}
\label{sec:ratpointsaffine}

\begin{df}
  \label{df:toponrationalpoints}
  Let $R$ be a sequential ring, and $X \rightarrow \spec R$ an affine scheme of finite type. Let $X = \spec A$. By using the topology on $R$, it is possible to topologize $X(R)$ as follows. There is a natural inclusion of sets
\begin{equation*}
  X(R) = \hom_{\Sch_R}(\spec R, X) = \hom_{\Alg_R}(A,R) \subset \hom_{\mathbf{Sets}}(A,R) = R^A.
\end{equation*}
  We endow $R^A$ with the product topology and provide $X(R) \subset R^A$ with the sequential saturation of the subspace topology. 
\end{df}

\begin{rmk}
  After taking a sequential saturation, the inclusion map
  \begin{equation*}
    X(R) \hookrightarrow R^A 
  \end{equation*}
  is continuous, but not an embedding. As we will see in \textsection \ref{sec:ratpointsgnral}, taking a sequential saturation in Definition \ref{df:toponrationalpoints} is important.
\end{rmk}

Every element $a \in A$ induces a map $\varphi_a: X(R) \rightarrow R$ by evaluating $R$-algebra homomorphisms $A \rightarrow R$ at $a$. Such a map agrees with the composition
\begin{equation*}
  X(R) \hookrightarrow R^A \rightarrow R
\end{equation*}
where the second map is given by projection to the $a$-th coordinate. By the previous remark, $\varphi_a$ is continuous.

There is another way to construct a topology on $X(R)$ which is more explicit. The choice of a closed embedding into an affine space identifies $X(R)$ with a subset of $R^n$ for some $n$, and we may endow $R^n$ with the product topology and $X(R) \subset R^n$ with the saturation of the subspace topology. If this procedure is taken, it is necessary to show that this topology on $X(R)$ does not depend on the choice of embedding into affine space. We explain how this works, and show that it is essentially equivalent to Definition \ref{df:toponrationalpoints}.

Let us choose an $R$-algebra isomorphism
\begin{equation}
  \label{eqn:chooseisom}
  A \simeq R\squarebs{t_1, \dots, t_n} /I
\end{equation}
and identify $X(R)$ with the set $V(I)$ of elements in $R^n$ on which all polynomial functions belonging to the ideal $I$ vanish. Consider the product topology on $R^n$ and endow $X(R)$ with the sequential saturation of the subspace topology. 

The choice of isomorphism corresponds to the choice of elements $a_i \in A$ that map to $t_i \pmod{I}$ under (\ref{eqn:chooseisom}), for $0 \leq i \leq n$. These induce a continuous map $R^A \rightarrow R^n$ by projecting to the coordinates indexed by $(a_1,\dots,a_n)$. The inclusion $X(R) \subset R^n$ factors then into
\begin{equation*}
  X(R) \hookrightarrow R^A \rightarrow R^n.
\end{equation*}

On one hand, the topology on $X(R)$ given by Definition \ref{df:toponrationalpoints} makes all inclusions $X(R) \hookrightarrow R^n$ given by choosing an $R$-algebra isomorphism continuous and hence provides a stronger topology.

On the other hand, assume that $X(R)$ is topologized according to an embedding into $R^n$ determined by (\ref{eqn:chooseisom}). Every element of $A$ is an $R$-polynomial in $a_1,\ldots,a_n$, and $R$ is a sequential ring. Hence, all polynomial maps $R^n \rightarrow R$ are sequentially continuous. It follows that the inclusion map $X(R) \hookrightarrow R^A$ is sequentially continuous. But on a sequential space a sequentially continuous map is necessarily continuous.

\begin{rmk}
  Over the affine line $\Ad_R^1 = \spec R[t]$, the topology we have just described on $\Ad_R^1(R) = R$ is the sequential saturation of the topology on $R$.
\end{rmk}

We summarize the construction discussed above in the statement below, along with some properties.

\begin{thm}
  \label{thm:topratptsaffine}
  Let $R$ be a sequential ring. There is a unique covariant functor
  \begin{equation*}
    \AffSch_R \to \Seq, \quad X \mapsto X(R)
  \end{equation*}
  which carries fibred products to products, closed immersions to topological embeddings and gives $\Ad_R^1\left( R \right)=R$ the sequential saturation of its topology.
  
  If $R$ is Hausdorff and $X \to \spec R$ is affine and of finite type, then $X(R)$ is Hausdorff and closed immersions $X \hookrightarrow X'$ induce closed embeddings $X(R) \hookrightarrow X'(R)$.
\end{thm}

\begin{proof}
  Regarding uniqueness, let $X \to \spec R$ be a finite type affine scheme. Choose a closed embedding $X \hookrightarrow \Ad_R^n$ for some $n\geq 0$. We look at the induced map on $R$-points: since we have compatibility with fibred products and we may view $\Ad^n$ as a product of $n$ copies of $\Ad^1$, and since closed immersions carry on to topological embeddings, $X(R) \subset R^n$ is an embedding into $R^n$ viewed as an object in $\Seq$. Hence, the topology is unique and $X(R)$ is Hausdorff whenever $R$ is Hausdorff, for a product of Hausdorff spaces in $\Seq$ is Hausdorff.

  Regarding existence, take the topology on $X(R)$ given by Definition \ref{df:toponrationalpoints}; we check the rest of the claimed properties. 
  
  If $X \rightarrow Y$ is a morphism of affine $R$-schemes, put $X = \spec A$, $Y = \spec B$. The homomorphism of $R$-algebras $B \rightarrow A$ induces a continuous map 
    \begin{equation}
      \label{eqn:pfaffinecase1}
      R^{A} \rightarrow R^{B}.
    \end{equation}
    Then, the commutative diagram
  \begin{equation*}
    \xymatrix{
    R^{A} \ar[r] & R^{B} \\
    X(R) \ar[u] \ar[r] & Y(R) \ar[u]
    }
  \end{equation*}
  shows that the natural map $X(R) \rightarrow Y(R)$ is continuous, and hence it is also continuous after taking sequential saturations.

  If $X \hookrightarrow Y$ is a closed immersion, then the map $B \rightarrow A$ is a surjective homomorphism of $R$-algebras and the map (\ref{eqn:pfaffinecase1}) is an embedding. From a topological point of view, this identifies $R^{A}$ with the subset of $R^{B}$ cut out by some equalities among components. This implies that when $R$ is Hausdorff, $X(R) \hookrightarrow Y(R)$ is a closed embedding of topological spaces.

  Finally, we check compatibility with fibred products. First, over the final object in the category: if $X$ and $Y$ are two affine schemes of finite type over $R$, we have an identification of sets
  \begin{equation*}
    \left( X \times_R Y \right)\left( R \right) = X\left( R \right) \times Y\left( R \right),
  \end{equation*}
  and both spaces are homeomorphic in $\Seq$, since we take the sequential saturation of the product topology on the left-hand side of the previous equation.

  In the general setting, assume that $X,Y$ and $Z$ are affine $R$-schemes of finite type and that we are given morphisms $X \rightarrow Z$, $Y \rightarrow Z$. Consider the isomorphism
  \begin{equation*}
    X \times_Z Y \simeq \left( X \times_R Y \right) \times_{Z \times_R Z} Z
  \end{equation*}
  and the topological homeomorphism that it induces. Since $Z$ is separated over $R$, and we already have compatibility over the final object in the category, we need only consider the case in which one of the projection maps is a closed immersion. Since closed immersions induce topological embeddings, we are done.
\end{proof}

Now we wish to study the behaviour of the topology defined in \ref{df:toponrationalpoints} with respect to base change. Let $R \rightarrow S$ be a morphism of sequential rings, as in Definition \ref{df:seqring}, and let $X \rightarrow \spec R$ be an affine scheme of finite type. We will identify
\begin{equation*}
  X(S) = X_S(S),
\end{equation*}
where $X_S = X \times_R S$. As $X_S \rightarrow \spec S$ is an affine scheme of finite type, the topology we will consider on $X_S(S)$ is that given by applying Definition \ref{df:toponrationalpoints} to $X_S \rightarrow \spec S$.

\begin{prop}
  \label{prop:basechange}
  Let $R \rightarrow S$ be a morphism of sequential rings and $X \rightarrow \spec R$ an affine scheme of finite type. Then the natural map $X(R) \rightarrow X(S)$ is continuous. Moreover, if $R \rightarrow S$ is an open (resp. closed) embedding, then $X(R) \rightarrow X(S)$ is an open (resp. closed) embedding.
\end{prop}

\begin{proof}
  Let $X = \spec A$, and let $X_S = \spec B$, where $B = A \otimes_R S$. We pick an isomorphism
  \begin{equation*}
    A \simeq R\squarebs{t_1, \ldots, t_n} / I
  \end{equation*}
  and identify $X(R) = V(I) \subset R^n$. After changing base over $S$, we have
  \begin{equation*}
    B \simeq S\squarebs{t_1, \ldots, t_n} / I^e,
  \end{equation*}
  where $I^e = IS\squarebs{t_1, \ldots, t_n}$ is the extension of the ideal $I$ along the ring homomorphism
  \begin{equation*}
    R\squarebs{t_1, \ldots t_n} \rightarrow S\squarebs{t_1, \ldots, t_n}.
  \end{equation*}
  We may identify $X(S) = X_S(S) = V(I^e) \subset S^n$. Now, the diagram
  \begin{equation}
    \label{eqn:diagramincaffine}
    \xymatrix{
    R^n \ar[r] & S^n \\
    V(I) \ar[r] \ar[u] & V(I^e), \ar[u]
    }
  \end{equation}
  whose vertical arrows are inclusions, shows that $X(R) \rightarrow X(S)$ is continuous when $X(R)$ and $X(S)$ are viewed as subspaces of $R^n$ and $S^n$. Therefore, we may take sequential saturations: the map $X(R) \rightarrow X(S)$ is continuous.

  Suppose now that $R \rightarrow S$ is a closed immersion. Then, $R^n \rightarrow S^n$ is also a closed immersion and, by restricting and taking sequential saturations, so is the map $X(R) \rightarrow X(S)$. This is also the case when we deal with an open immersion.
\end{proof}

A construction which may also be considered in this setting is the Weil restriction. Let $R \rightarrow S$ be an extension of rings, and $Y \rightarrow \spec S$ be an affine scheme of finite type. Consider the functor $X : \Sch_R \rightarrow \mathbf{Sets}$ defined by
\begin{equation*}
  X(T) = \mathrm{Hom}_{\Sch_S}(T \times_R S, Y), \quad T \in \mathrm{Ob}(\Sch_R).
\end{equation*}
Whenever $X$ is representable, the associated scheme is called the Weil restriction of $Y$ along $R \rightarrow S$.

Assume that $R \hookrightarrow S$ is an injective morphism of sequential rings such that $S$ is a finite type, locally free $R$-module and the topology on $S$ is the quotient topology from a presentation (equivalently, any presentation) as a quotient of a finite type free $R$-module. In particular, as $S$ is a projective $R$-module, the inclusion map $R \rightarrow S$ admits an $R$-linear splitting and $R$ may be viewed as a subspace of $S$.

\begin{exa}
  A finite extension of higher local fields $F \hookrightarrow L$ satisfies the above conditions, and so does the extension of valuation rings $\OO_F \hookrightarrow \OO_L$.
\end{exa}

The conditions stated above are enough to guarantee the existence of the Weil restriction of $Y \rightarrow \spec S$ along $R \hookrightarrow S$ \cite[\textsection 7.6]{bosch-neron-models}.

\begin{prop}
  Let $R \hookrightarrow S$ be an injective morphism of sequential rings, such that $S$ is a locally free, finite type $R$-module. Let $Y \rightarrow \spec S$ be an affine scheme of finite type . The Weil restriction $\mathcal{R}_{S|R}(Y) = X \rightarrow \spec R$ is an affine scheme of finite type. By definition,
  \begin{equation*}
    X(R) = \mathrm{Hom}_{\Sch_S}(\spec R \times_R S, Y) = \mathrm{Hom}_{\Sch_S}(\spec S, Y) = Y(S).
  \end{equation*}
  The two topologies we have defined on this set agree.
\end{prop}

\begin{proof}
  The reason why this is true is because the topologies on $X(R)$ and $Y(S)$ already coincide before taking sequential saturations. We reproduce the argument of Conrad \cite[Example 2.4]{conrad-weil-grothendieck-approaches-to-adelic-points}.

  As Weil restriction respects closed immersions, it is enough to prove the result when $Y$ is an affine space.

  Let $P$ be a free $R$-module such that $S$ is a direct summand of it. By duality, we have a surjection of $R$-modules
  \begin{equation*}
    P^* = \mathrm{Hom}_R(P, R) \twoheadrightarrow \mathrm{Hom}_R(S, R) = S^*.
  \end{equation*}

  Let $A$ be any $R$-algebra. The natural map $S \otimes_R A \rightarrow P \otimes_R A$ is injective and functorially defined by a system of $R$-linear equations in $A$.

  Now let $M = R^n$, such that $Y = \spec\left( \mathrm{Sym}_S(M') \right)$, with $M' = M \otimes_R S$. Then $X$ is in a natural way a closed subscheme of $\spec\left( \mathrm{Sym}_R(M \otimes_R P^*) \right)$.

  Then $Y(S) = \mathrm{Hom}_S(M', S) = \mathrm{Hom}_R(M,S)$ has a natural topology as a finite type free $S$-module. Because of the inclusion $S \hookrightarrow P$, we may regard
  \begin{equation*}
    X(R) = \mathrm{Hom}_R(M,S) = M^* \otimes_R S \hookrightarrow M^* \otimes_R P
  \end{equation*}
  as a subspace.

  The topologies on $X(R)$ and $Y(S)$ agree if $S$ is a subspace of $P$. But $S$ is a direct summand of $P$ and the subspace topology agrees with the quotient topology given by the surjection $P \twoheadrightarrow S$, which is the original topology on $S$ by hypothesis.

  Since the two topologies on $X(R) = Y(S)$ agree, they are also equal after taking sequential saturations.
\end{proof}

\subsection{The general case}
\label{sec:ratpointsgnral}

In order to adapt the construction explained in \textsection\ref{sec:ratpointsaffine} to general schemes of finite type $X \rightarrow \spec R$, with $R$ a sequential ring, our argument has two steps. First, we take an affine open cover of $X$ and use Theorem \ref{thm:topratptsaffine} to topologize the affine open sets in the cover. Second, we need a compatibility condition in order to be able to define a topology on $X$.

There are a few obstacles before we can apply this argument, which will be sorted out by assuming further properties on the sequential ring $R$. Namely, we need to show that for any open immersion of affine schemes $U \rightarrow X$, we have an open embedding of topological spaces $U(R) \hookrightarrow X(R)$. For a general ring $R$ (even for a topological ring) this property will not hold in general and $U(R) \rightarrow X(R)$ may even fail to be a topological embedding.

There is one elementary example that illustrates this situation. Suppose that $R$ is a topological ring and consider $U = \mathbb{G}_m$ as the complement of the origin in the affine line $\mathbb{A}_R^1$. The map $U(R) \rightarrow \Ad_R^1(R)$ is the inclusion $R^\times \rightarrow R$, which is not an embedding unless inversion on $R^\times$ is continuous, since the topology on $\mathbb{G}_m$ is given by Theorem \ref{thm:topratptsaffine}.

From this, we gather two necessary conditions that need to be imposed on a topological ring $R$, if we want to be able to use our argument. The first of these conditions is continuity of the inversion map, so that $R^\times \rightarrow R$ is an embedding. The second condition is that $R^\times$ is an open subset of $R$, so that the embedding is open.

For a sequential ring, the above example gives the inclusion $R^\times \subset R$, where both sets have the topologies described in \textsection \ref{sec:ratpointsaffine}, which are sequentially saturated. In order for this inclusion to be an open embedding, it is enough to require inversion on $R$ to be sequentially continuous and the subset $R^\times$ to be open in $R$.

There is another aspect we need to worry about. Whenever $X = \cup_i U_i$ is an affine cover, we need $X(R)$ to be covered by the subsets $U_i(R)$. This will be the case whenever $R$ is local.

\begin{rmk}
  The three conditions imposed on $R$ (sequential continuity of inversion, openness of the unit group and being local) are analogous to the restrictions for the similar argument to work when $R$ is a topological ring \cite[Proposition 3.1]{conrad-weil-grothendieck-approaches-to-adelic-points}.
\end{rmk}

The above conditions are sufficient to extend Theorem \ref{thm:topratptsaffine} to general schemes of finite type.

\begin{thm}
  \label{thm:topratptsgeneral}
  Let $R$ be a local sequential ring such that $R^\times \subset R$ is open and such that inversion $R^\times \rightarrow R^\times$ is sequentially continuous. There is a unique covariant functor
  \begin{equation*}
    \Sch_R \to \Seq
  \end{equation*}
  which carries open (resp. closed) immersions of schemes to open (resp. closed) topological embeddings, fibred products to products, and giving $\mathbb{A}_R^1(R) = R$ the sequential saturation of its topology.

  This agrees with the construction in Theorem \ref{thm:topratptsaffine} for affine schemes. If $R$ is Hausdorff and $X \rightarrow \spec R$ is separated and of finite type, then $X(R)$ is Hausdorff.
\end{thm}

\begin{proof}
  We remark that under the hypotheses on $R$, the inclusion $R^\times \hookrightarrow R$ is an open embedding in $\Seq$.

  The key to the argument is showing that whenever $U \rightarrow X$ is an open immersion of affine schemes of finite type over $R$, then $U(R) \hookrightarrow X(R)$ is an open embedding.
  It is enough to show this in the case in which $U$ is a principal open set, so $X = \spec A$, $U = \spec A_f$, for some $f \in A$. Then, $U(R)$ is the preimage of $R^\times$ under the map $X(R) \rightarrow R$ associated to $f$. The fibre square
  \begin{equation}
    \label{eqn:topratptsgeneral}
    \xymatrix{
    U \ar[r]^f \ar[d] & \mathbb{G}_m \ar[d] \\
    X \ar[r] & \Ad_R^1    
    }
  \end{equation}
  reduces the problem to the special case $U = \mathbb{G}_m$, $X = \Ad_R^1$, because all schemes involved in (\ref{eqn:topratptsgeneral}) are affine and, in such case, we already have compatibility with fibred products. But, after saturating, $R^\times \rightarrow R$ is an open embedding by hypothesis.

  If $X = \cup_i U_i$ is an open cover, then $X(R) = \cup_i U_i\left( R \right)$, because a map $\spec R \rightarrow X$ that carries the closed point into $U_i$ must land entirely inside $U_i$, since the only open subscheme of $\spec R$ that contains the closed point is the entire space.
\end{proof}

\begin{rmk}
  Let $R$ be a sequential ring and $G \rightarrow \spec R$ a finite type group scheme. The structure morphisms
  \begin{align*}
    \mu: G \times_R G \longrightarrow G, \\
    \varepsilon: \spec R \longrightarrow G, \\
    \iota: G \longrightarrow G,
  \end{align*}
  together with their properties, turn the set $G(R)$ into a group by taking $R$-points. Since the topology we have constructed is functorial, $G(R)$ becomes a topological group.
\end{rmk}

\begin{rmk}
  Let $F$ be a higher dimensional local field equipped with a higher topology. On one hand, as any higher topology is Hausdorff, $F^\times$ is open in $F$, as it is the complement of a singleton. On the other hand, inversion is sequentially continuous. Finally, a field is a local ring. It is because of this that $F$ satisfies the hypotheses of Theorem \ref{thm:topratptsgeneral}.
\end{rmk}

\begin{rmk}
  Since the residue map $\rho: \OO_F \rightarrow \overline{F}$ is continuous and the topology on $\overline{F}$ is Hausdorff, the maximal ideal $\pp_F = \rho^{-1}(\left\{ 0 \right\})$ is a closed set. Since $\OO_F$ is local, $\OO_F^\times = \OO_F \setminus \pp_F$ is open in $\OO_F$. Inversion on $\OO_F^\times$ is sequentially continuous, and this means that Theorem \ref{thm:topratptsgeneral} may also be applied to the ring $\OO_F$.
\end{rmk}

\begin{rmk}
  By iteration of the argument in the previous remark, Theorem \ref{thm:topratptsgeneral} may also be applied to $O_F$ and to any of the higher rank valuation rings of $F$, displayed in (\ref{eqn:higherranksubrings}).
\end{rmk}

\begin{rmk}
  For the reasons exposed above, the results in this section also apply to archimedean higher local fields, higher fields over perfect fields and their higher rank rings of integers.
\end{rmk}

Regarding the behaviour of this topology under base change, Proposition \ref{prop:basechange} holds if we remove the affineness condition on $X$. The reason is that in order to check the conditions stated, we may restrict ourselves to an affine open subscheme, for which Proposition \ref{prop:basechange} is valid. For the sake of completeness, we state this result as a proposition.

\begin{prop}
  \label{prop:basechangegnrl}
  Let $R \rightarrow S$ be a morphism of sequential rings with $R$ and $S$ satisfying the hypotheses of Theorem \ref{thm:topratptsgeneral}. Let $X \rightarrow \spec R$ be a scheme of finite type. Then, the map $X(R) \rightarrow X(S)$ is continuous. Moreover, if $R \rightarrow S$ is an open (resp. closed) immersion, then $X(R) \rightarrow X(S)$ is also an open (resp. closed) immersion.
\end{prop}

\section{Future work}
\label{sec:future}

\paragraph{Reduction maps.} These are a particular case of base change. For a higher local field $F$ of dimension $n \geq 1$ endowed with a higher topology, $\overline{F}$ also is endowed with a higher topology (if $\overline{F}$ is finite, we consider the discrete topology), and the topologies on $F$ and $\overline{F}$ are compatible.

  If $X \to \spec \OO_F$ is of finite type, flat and irreducible, then the reduction map
  \begin{equation*}
    \rho: X(\OO_F) \rightarrow X_{\overline{F}}(\overline{F})
  \end{equation*}
  is surjective \cite[Prop. 10.1.36]{liu-algebraic-geometry}. Besides the map $\rho$ being continuous by Proposition \ref{prop:basechangegnrl}, one can say more in certain cases. For example, if $X = \Ad^1_{\OO_F}$, then $\rho$ is open: we are dealing with the map $\rho: \OO_F \to \overline{F}$ with respect to the sequential topologies and, since these are group topologies, it is enough to check that the image of a sequentially open neighbourhood of zero in $\OO_F$ is sequentially open in $\overline{F}$. The obvious argument works because we are able to choose very particular lifts of elements of a sequence converging to zero in $\overline{F}$. The same argument could be adapted to the case $X=\Ad^n_{\OO_F}$.

  A general argument along these lines would not be possible: affine schemes may be viewed as closed embeddings into affine spaces and the restriction of an open map to a closed subspace is not necessarily open.
  
  It would be very interesting to know if the two topologies on either side of the map $\rho$ are related. More precisely: is $\rho$ open, perhaps under certain conditions on $X$? If not, is it at least a quotient mapping?

  \paragraph{Relation between topology and integration.} One of the goals of the higher adelic programme is to generalize the techniques used for the study of local fields to higher local fields. In this direction, Fesenko \cite{fesenko-aoas1} and Morrow \cite{morrow-integration-on-valuation-fields}, \cite{morrow-integration-on-product-spaces} have introduced and developed a two-dimensional measure and a theory of harmonic analysis on two-dimensional local fields, generalising the local results of \cite{tate-thesis}. The global counterpart to this theory has been developed by Fesenko in \cite{fesenko-aoas2}.

The extension of these methods to sets of rational points over higher local fields and the development of a new integration theory on such spaces would be a very important achievement in the direction of a better understanding of the arithmetic of higher dimensional schemes.

More precisely, harmonic analysis on rational points over higher local fields would be very helpful for the study of sets of rational points over higher adelic rings and establishment of a theory of higher dimensional Tamagawa numbers.

It seems that a good understanding of the topology of these spaces, as well as the relation between topology and two-dimensional measure, could be an important contribution in this direction.

\paragraph{Points over higher adelic rings.} There are several higher adelic objects which may be defined as a restricted product of higher local fields.

In the work of Fesenko in dimension two \cite{fesenko-aoas2} there are several rings that may be realised as a restricted product of higher local fields and rings. 

In Beilinson's general simplicial approach to higher adeles on noetherian schemes \cite{huber-parshin-beilinson-adeles}, some of the adelic objects described are also definable in terms of restricted products of higher local fields.

It seems that an approach to endowing sets of rational points over such higher adelic rings with topologies would require the use of the results in this work as a starting point.

\paragraph{Representation theory of algebraic groups over two-dimensional local fields.} Such representation theory has been developed by Gaitsgory-Kazhdan \cite{gaitsgory-kazhdan-reps-alg-grps-over-2ldf} and Braverman-Kazhdan \cite{braverman-kazhdan-examples-hecke-algs-for-2ldfs} among other works.

Algebraic groups over two dimensional local fields and their central extensions are a generalization of formal loop groups and are related to a generalization of a class of affine Kac-Moody groups. Hence the results in this work can find applications in the corresponding representation theory.

\vskip .5cm

\def\cprime{$'$} \def\cprime{$'$}

\end{document}